\documentclass[12pt,reqno]{amsart}
\usepackage{amsmath, amsfonts, amssymb, amsthm}
\def\Z{\mathbb Z}

\def\N{\mathbb N}

\def\1{{\bf 1}}

\theoremstyle{plain}
\newtheorem{theorem}{Theorem}[section]

\theoremstyle{definition}

\theoremstyle{remark}

\begin{document}
\title{A Stern-type congruence for the Schr\"oder numbers}

\begin{abstract}
For the Schr\"oder number
$$
S_n=\sum_{k=0}^n\binom{n}k\binom{n+k}k\frac1{k+1},
$$
we prove that
$$
S_{n+2^\alpha}\equiv S_{n}+2^{\alpha+1}\pmod{2^{\alpha+2}},
$$
where $n\geq 1$ and $\alpha\geq 1$.
\end{abstract}
\author{Hui-Qin Cao} \author{Hao Pan}
\address{Department of Applied Mathematics, Nanjing Audit University, Nanjing 210029, People's Public of China}
\email{caohq@nau.edu.cn}
\address{Department of Mathematics, Nanjing University, Nanjing 210093,
People's Republic of China}
\email{haopan79@zoho.com}
\keywords{Congruence; Schr\"oder number}

\subjclass[2010]{Primary 11B65; Secondary 05A10, 11A07}
\maketitle

In combinatorics, the Schr\"oder number $S_n$ is the number of all lattice paths from $(0,0)$ and $(n,n)$
which just use the steps $(1,0)$, $(0,1)$, $(1,1)$ and contain no point above the line $y=x$.
Using a simple combinatorial discussion, we can easily prove that
$$
S_n=\sum_{k=0}^n\binom{n+k}{2k}\binom{2k}{k}\frac1{k+1}.
$$
In fact, assuming a Schr\"oder path exactly uses $n-k$ steps $(1,1)$, there are $\binom{n+k}{n-k}$ ways to insert those
$n-k$ steps $(1,1)$ between $k$ steps $(1,0)$ and $k$ steps $(0,1)$. Also, it is well-known that
the number of the lattice paths from $(0,0)$ and $(k,k)$,
which use the steps $(1,0)$, $(0,1)$ and never rise above the line $y=x$, is the Catalan number $C_k=\frac1{k+1}\binom{2k}{k}$.
Furthermore, we know that the generating function for the Schr\"oder numbers is
$$
\sum_{n=0}^\infty S_nx^n=\frac{1-x-\sqrt{1-6x+x^2}}{2x}.
$$
Nowadays, the combinatorial properties of the Schr\"oder numbers have been widely discussed.
For example, in \cite[Exercise 6.39]{Stanely99}, Stanley listed many different combinatorial interpretations for $S_n$.

However, seemingly the arithmetical properties of $S_n$ were seldom studied before.
In fact, we know that $S_n$ is always even for $n\geq 1$. The number
$$
s_n=\frac{1}{2}S_n
$$
is also called the little Schr\"oder number, which counts the number of ways of dissecting a convex polygon with $n+1$ sides into smaller polygons by inserting diagonals.
And in \cite{Sun11}, Sun obtained several interesting congruences on the sum $\sum_{k=0}^{p-1}S_km^{-k}$ modulo odd prime $p$.

On the other hand, the Euler number $E_n$ is given by
$$
\sum_{n=0}^\infty\frac{E_n}{n!}x^n=\frac{2}{e^x+e^{-x}}.
$$
Clearly $E_{n}=0$ if $n$ is odd.
The Euler numbers play an important role in number theory, since $E_{2n}/2$ coincides with the Dirichlet $L$-function $L(-2n,\chi_{-4})$.
As early as 1851, Kummer had proved
\begin{equation}\label{kummereuler}
E_{2n+(p-1)}\equiv E_{2n}\pmod{p}
\end{equation}
for any odd prime $p$.
Furthermore, in combinatorics, $(-1)^nE_{2n}$ is the number of all alternative permutations on $\{1,2,\ldots,2n\}$, i.e.,
$$
(-1)^nE_{2n}=\#\{\sigma:\,\sigma(1)<\sigma(2)>\sigma(3)<\cdots>\sigma(2n-1)<\sigma(2n)\}.
$$
In 1875, Stern proved another curious identity for the Euler numbers:
\begin{equation}\label{sterneuler}
E_{2n+2^\alpha}\equiv E_{2n}+2^\alpha\pmod{2^{\alpha+1}},
\end{equation}
where $\alpha\geq 1$.
Clearly Stern's congruence implies that $E_{2n}\equiv E_{2m}\pmod{2^{\alpha}}$ if and only if $2^\alpha$ divides $2n-2m$.
For the history and a proof of Stern's congruence, the reader may refer to Wagstaff's expository article \cite{Wagstaff02}.

Obviously from the viewpoint of either combinatorics or number theory, the Schr\"oder numbers have no connection with the Euler numbers. However, in this short note, we shall show that $S_n$ also satisfies a very similar congruence as (\ref{sterneuler}).
\begin{theorem}
For any $n\geq 1$ and $\alpha\geq 1$,
\begin{equation}\label{sternschroder}
S_{n+2^\alpha}\equiv S_{n}+2^{\alpha+1}\pmod{2^{\alpha+2}}.
\end{equation}
\end{theorem}
Clearly (\ref{sternschroder}) is equivalent to
\begin{equation}\label{sternlittleschroder}
s_{n+2^\alpha}\equiv s_{n}+2^{\alpha}\pmod{2^{\alpha+1}},
\end{equation}
i.e., $s_{n}\equiv s_{m}\pmod{2^{\alpha}}$ if and only if $2^\alpha$ divides $n-m$.
In fact, we can prove a stronger result.
\begin{theorem}
For any $n\geq 1$ and $\alpha\geq 2$,
\begin{equation}\label{sternschroder2}
S_{n+2^\alpha}\equiv S_{n}+(-1)^{\lfloor\frac{n-1}{2}\rfloor}2^{\alpha+1}\pmod{2^{\alpha+3}},
\end{equation}
where $\lfloor x\rfloor=\max\{n\in\Z:\,n\leq x\}$. And when $\alpha=1$, we also have
\begin{equation}\label{sternschroder21}
S_{n+2}\equiv S_{n}+4\pmod{16}.
\end{equation}
\end{theorem}
\begin{proof}
We need another explicit form of $S_n$:
\begin{equation}\label{schrodertwo}
S_n=\sum_{k=1}^n\frac1n\binom{n}{k}\binom{n}{k-1}2^k.
\end{equation}
Of course, (\ref{schrodertwo}) is known. However, for the sake of the completeness, here we give a proof:
\begin{align*}
&\sum_{k=1}^n\frac1n\binom{n}{k}\binom{n}{k-1}2^k=
\sum_{k=1}^n\frac1n\binom{n}{k}\binom{n}{k-1}\sum_{j=0}^k\binom{k}{j}\\
=&
\sum_{j=0}^n\frac1n\binom{n}{j}\sum_{k=j}^n\binom{n-j}{k-j}\binom{n}{k-1}
=\sum_{j=0}^n\frac1n\binom{n}{j}\binom{2n-j}{n-1}\\
=&
\sum_{k=0}^n\frac1n\binom{n}{k}\binom{n+k}{n-1}=
\sum_{k=0}^n\frac1{k+1}\binom{n+k}{2k}\binom{2k}{k},
\end{align*}
where the Chu-Vandemonde identity is used in the third step.

Denote $$T(n,k)=\frac{1}{n}\binom{n}{k}\binom{n}{k-1}2^k.$$
In fact, $N(n,k)=T(n,k)/2^k$ is also called the Narayana number, which counts the number of the Calatan paths with $k$ peaks from $(0,0)$ to $(n,n)$. 
Then $$S_n=\sum_{k=1}^nT(n,k).$$ It is clear that $$T(n, k)\equiv 0\pmod{2^k},$$
since the Narayana number
$$
\frac{1}{n}\binom{n}{k}\binom{n}{k-1}=\binom{n+1}{k}\binom{n-1}{k-1}-\binom{n}k\binom{n}{k-1}
$$
is an integer.

For $\alpha\geq 1$, 
clearly
\begin{align}
T(n+2^{\alpha}, 1)=\frac{1}{n+2^{\alpha}}\binom{n+2^{\alpha}}{1}\binom{n+2^{\alpha}}{0}\cdot 2
=\frac{1}{n}\binom{n}{1}\binom{n}{0}\cdot 2=T(n, 1)\label{T1},
\end{align}
and
\begin{align}
T(n+2^{\alpha}, 2)=&\frac{1}{n+2^{\alpha}}\binom{n+2^{\alpha}}{2}\binom{n+2^{\alpha}}{1}\cdot 2^2\notag\\
=&2(n(n-1)+2^{\alpha+1}n-2^{\alpha}+2^{2\alpha})\notag\\
=&T(n, 2)+2^{\alpha+2}n-2^{\alpha+1}+2^{2\alpha+1}.\label{T2}
\end{align}
Moreover, noting that
\begin{equation}\label{binomial2alpha}
\binom{n+2^{\alpha}}{2}=\binom{n-1}{2}+\binom{n-1}{1}\binom{2^\alpha}{1}+\binom{2^\alpha}{2}
\equiv\binom{n-1}{2}+2^{\alpha-1}\pmod{2^{\alpha}},
\end{equation}
we get
\begin{align}
T(n+2^{\alpha},3)
=&\frac{1}{3}\binom{n+2^{\alpha}-1}{2}\binom{n+2^{\alpha}}{2}\cdot 2^3\notag\\
\equiv&\frac{2^3}{3}\bigg(\binom{n-1}{2}+2^{\alpha-1}\bigg)\bigg(\binom{n}{2}+2^{\alpha-1}\bigg)\notag\\
=&\frac{2^3}{3}\bigg(\binom{n-1}{2}\binom{n}{2}+(n-1)^22^{\alpha-1}+2^{2\alpha-2}\bigg)\notag\\
\equiv&T(n, 3)+(n-1)2^{\alpha+2}+2^{2\alpha+1}\pmod{2^{\alpha+3}}\label{T3}.
\end{align}
Thus, when $\alpha=1$, we have 
\begin{align*}
S_{n+2}=&\sum_{k=1}^{n+2}T(n+2, k)\equiv\sum_{k=1}^3T(n+2, k)\pmod{2^4}\\
\equiv&T(n, 1)+\left(T(n, 2)+n2^3-2^2+2^3\right)+\left(T(n, 3)+(n-1)2^3+2^3\right)\\
\equiv&\sum_{k=1}^3T(n, k)+(2n-1)2^3-2^2\equiv S_n+2^2\pmod{2^4}.
\end{align*}

Below assume that $\alpha\geq 2$. Since $2^{\alpha}\geq\alpha+2$, we get $$T(n, k)\equiv T(n+2^{\alpha}, k)\equiv 0\pmod{2^{\alpha+3}}$$ for $k>2^{\alpha}$. So
\begin{align*}
S_{n+2^{\alpha}}=\sum_{k=1}^{n+2^{\alpha}}T(n+2^{\alpha}, k)\equiv\sum_{k=1}^{2^{\alpha}}T(n+2^{\alpha}, k)\pmod{2^{\alpha+3}}.
\end{align*}
By the Chu-Vandemonde identity,  
$$
\binom{n+2^{\alpha}}
{k}=\sum_{j=0}^{k}\binom{n}{k-j}\binom{2^{\alpha}}{j}.
$$
Since $$\binom{2^{\alpha}}{j}=\frac{2^{\alpha}}{j}\binom{2^{\alpha}-1}{j-1}\equiv 0\pmod{2^{\alpha-\nu_2(j)}},$$ 
where $$\nu_2(j)=\max\{b\in\N:\,2^b\mid j\}.$$
Clearly $\nu_2(j)\leq\lfloor\log_2k\rfloor$ 
for any $1\leq j\leq k$. We obtain that
$$\binom{n+2^{\alpha}}{k}\equiv\binom{n}{k}\pmod{2^{\alpha-\lfloor\log_2k\rfloor}}.$$
If $k$ is odd, then
\begin{align*}
T(n+2^{\alpha}, k)=&\frac{1}{n+2^{\alpha}}\binom{n+2^{\alpha}}{k}\binom{n+2^{\alpha}}{k-1}\cdot 2^k\\
=&\frac{2^k}{k}\binom{n-1+2^{\alpha}}{k-1}\binom{n+2^{\alpha}}{k-1}\\
\equiv&\frac{2^k}{k}\binom{n-1}{k-1}\binom{n}{k-1}=T(n, k)\pmod{2^{\alpha+k-\lfloor\log_2k\rfloor}}.
\end{align*}
And if $k$ is even, then
\begin{align*}
T(n+2^{\alpha},k)=&\frac{1}{n+2^{\alpha}}\binom{n+2^{\alpha}}{k}\binom{n+2^{\alpha}}{k-1}\cdot 2^k\\
=&\frac{2^k}{k-1}\binom{n+2^{\alpha}}{k}\binom{n-1+2^{\alpha}}{k-2}\\
\equiv&\frac{2^k}{k-1}\binom{n}{k}\binom{n-1}{k-2}=T(n, k)\pmod{2^{\alpha+k-\lfloor\log_2k\rfloor}}.
\end{align*}
Thus $$T(n+2^{\alpha}, k)\equiv T(n, k)\pmod{2^{\alpha+k-\lfloor\log_2k\rfloor}}$$ for each $1\leq k\leq 2^{\alpha}$. It is easy to verify $k-\lfloor\log_2k\rfloor\geq 3$ when $k\geq 5$. Thus we have
\begin{align*}
S_{n+2^{\alpha}}\equiv&\sum_{k=1}^4T(n+2^{\alpha}, k)+\sum_{k=5}^{2^{\alpha}}T(n+2^{\alpha}, k)\\
\equiv&\sum_{k=1}^4 T(n+2^{\alpha}, k)+\sum_{k=5}^{2^{\alpha}}T(n, k)\pmod{2^{\alpha+3}}.
\end{align*}

Clearly $$
\binom{2^{\alpha}}{j}=\frac{2^{\alpha}}{j}\binom{2^{\alpha}-1}{j-1}\equiv 0\pmod{2^{\alpha-1}}$$ for $1\leq j\leq 3$. 
And
$$
\binom{2^{\alpha}}{4}=2^{\alpha-2}\binom{2^{\alpha}-1}{3}\equiv 2^{\alpha-2}\pmod{2^{\alpha-1}}.
$$
So
$$
\binom{n+2^{\alpha}}{4}=\sum_{j=0}^4\binom{n}{4-j}\binom{2^{\alpha}}{j}\equiv\binom{n}{4}+2^{\alpha-2}\pmod{2^{\alpha-1}}.
$$
Also, by (\ref{binomial2alpha}),
$$
\binom{n-1+2^{\alpha}}{2}\equiv\binom{n-1}{2}\pmod{2^{\alpha-1}}.
$$
It follows that
\begin{align}
T(n+2^{\alpha}, 4)=&\frac{1}{3}\binom{n+2^{\alpha}}{4}\binom{n-1+2^{\alpha}}{2}\cdot 2^4\notag\\
\equiv &\frac{2^4}{3}\left(\binom{n}{4}+2^{\alpha-2}\right)\binom{n-1}{2}\notag\\
\equiv &T(n, 4)+\binom{n-1}{2}2^{\alpha+2}\pmod{2^{\alpha+3}}\label{T4}.
\end{align}
Combining (\ref{T1}), (\ref{T2}), (\ref{T3}) and (\ref{T4}), we get
\begin{align*}
\sum_{k=1}^4T(n+2^{\alpha}, k)
\equiv &T(n,1)+\left(T(n, 2)+n\cdot 2^{\alpha+2}-2^{\alpha+1}+2^{2\alpha+1}\right)\\
&+\left(T(n, 3)+(n-1)2^{\alpha+2}+2^{2\alpha+1}\right)
+\left(T(n, 4)+\binom{n-1}{2}2^{\alpha+2}\right)\\
\equiv&\sum_{k=1}^4 T(n,k)+2^{\alpha+1}(n^2+n-1)\\
\equiv&\sum_{k=1}^4 T(n,k)+(-1)^{\lfloor\frac{n-1}{2}\rfloor}2^{\alpha+1}\pmod{2^{\alpha+3}}.
\end{align*}
Consequently,
\begin{align*}
S_{n+2^{\alpha}}\equiv&\sum_{k=1}^4T(n,k)+(-1)^{\lfloor\frac{n-1}{2}\rfloor}2^{\alpha+1}+\sum_{k=5}^{2^{\alpha}}T(n, k)\\
=&S_n+(-1)^{\lfloor\frac{n-1}{2}\rfloor}2^{\alpha+1}\pmod{2^{\alpha+3}}.
\end{align*}
\end{proof}

\end{document}